\documentclass[final]{siamltex}
\newtheorem{example}{Example}[section]

\newcommand{\R}{\rm I\kern-.19emR}
\newcommand{\bQ}{\rm I\kern-.19emQ}
\newcommand{\C}{\rm I\kern-.17emC}
\newcommand{\bR}  {\R}
\newcommand{\bRn}  {\R^{n,n}}

\newcommand{\beqo} {\begin {eqnarray*}}
\newcommand{\beq} {\begin {eqnarray}}
\newcommand{\eeqo} {\end {eqnarray*}}
\newcommand{\eeq} {\end {eqnarray}}

\newcommand{\bdefi} {\begin {definition}}
\newcommand{\edefi} {\end {definition}}
\newcommand{\bpro} {\begin {proposition}}
\newcommand{\epro} {\end {proposition}}
\newcommand{\btheo} {\begin {theorem}}
\newcommand{\etheo} {\end {theorem}}
\newcommand{\blem} {\begin {lemma}}
\newcommand{\elem} {\end {lemma}}
\newcommand{\bcor} {\begin {corollary}}
\newcommand{\ecor} {\end {corollary}}

\title {Hidden decay in the inverses of acyclic matrices}

\author{Reinhard Nabben\thanks{
     	TU Berlin 
        Germany, email: nabben@math.tu-berlin.de, date of this version: 27.06.2026}}

\begin{document}

\maketitle

\begin{abstract}
It is well-known that  the absolute values of the entries of the inverse  of a diagonally dominant tridiagonal matrix  decay away from the diagonal. Moreover, lower and upper bounds for these  absolute values are known for a long time.  The graph of a tridiagonal matrix can be viewed as a  line. Thus  tridiagonal matrices are special acyclic matrices, i.e. matrices whose graph is a tree. 
In this paper we explore the structure of inverses of acyclic matrices.
However, there is no strict decay from the diagonal in the  inverses of diagonally dominant acyclic matrices. The decay is somehow hidden.  Here we identify this  decay which can be described nicely.
Moreover  we give bounds  for the absolute value of the entries  of the inverses of acyclic matrices. All of our results generalize results for tridiagonal matrices. 
\end{abstract}



\pagestyle{myheadings}
\thispagestyle{plain}

\markboth{R. NABBEN}{Hidden decay in the inverses of acyclic matrices } 

\section{Introduction}


An $n \times n$ matrix is called {\it sparse} if the number of its nonzero entries is much smaller than $n^2$. Sparsity reduces the storage requirements and the complexity of matrix-vector multiplications. Hence,  sparsity is useful
in almost all branches of numerical analysis. An 
$n \times n$ matrix is called {\it data sparse} if it can be represented by much less than  $n^2$ parameters with respect to a certain format or structure. Nevertheless, this format or structure is often hidden and must be identified. Thus, exploiting  hidden structure in matrix computation  
is a key topic in numerical analysis right now \cite{BenS16}.

The most prominent class of  data sparse matrices is the class of inverses of non-singular tridiagonal matrices. Non-singular tridiagonal $n \times n$ matrices can be described by $3n-2$ parameters. The inverses are dense but  can also be described by  $3n-2$ parameters. Gantmacher and Krein \cite{GK1} and \cite{GK2}
proved that 
a symmetric, irreducible non-singular matrix $A$ is tridiagonal 
if and only if $A^{-1} =: C =  [c_{i,j}]$
is given by two sequences 
$\{u_i\}_{i=1}^{n}$, $\{v_i\}_{i=1}^{n}$ of numbers such that
\beq \label{uv1}
C = \left[ \begin{array}{cccc}
u_1v_1 & u_1v_2 & \cdots & u_1v_n  \\
u_1v_2 & u_2v_2& \cdots & u_2v_n \\
 \vdots & \vdots & \ddots & \vdots \\
u_1v_n & u_2v_n & \cdots & u_nv_n
\end{array} \right] \quad \mbox{i.e.} \quad 
c_{i,j} = \left\{ \begin{array}{lr}

u_iv_j & i \leq j,\\
u_jv_i & i \geq j.
\end{array} \right. 
\eeq

Matrices of the form (\ref{uv1}) are called Green's matrices by Gantmacher and Krein \cite{GK2}. It was observed in \cite{McDNNST9} that 
Green's matrices can be described more elegantly 
 as the Hadamard
product (element-wise product) of a so-called type D matrix and a flipped type D matrix: 

\beq \label{inva11}
C = \left[ \begin{array}{cccc}
u_1 & u_1 & \cdots & u_1  \\
u_1 & u_2 & \cdots & u_2 \\
 \vdots & \vdots & \ddots & \vdots \\
u_1 & u_2 & \cdots & u_n
\end{array} \right] \circ
\left[ \begin{array}{cccc}
v_1 & v_2 & \cdots & v_n  \\
v_2 & v_2 & \cdots & v_n \\
 \vdots & \vdots & \ddots & \vdots \\
v_n & v_n & \cdots & v_n
\end{array} \right].
\eeq

A similar result holds for the inverse of a non-symmetric irreducible 
tridiagonal matrix. 
There the inverse $C = A^{-1} $ can be described by four   sequences $\{u_i\},\{v_i\},\{x_i\},\{y_i\}$ which satisfy $u_iv_i = x_iy_i$ (\cite{I}). In detail we have  
\beq \label{uvxy11}
c_{i,j} = \left\{ \begin{array}{lr}
u_iv_j & i \leq j,\\
x_iy_j & i \geq j.
\end{array} \right.
\eeq

As in the symmetric case matrices of the form (\ref{uvxy11}) can be written nicely as 
the Hadamard product of two matrices:
\beq \label{uninva1}
C = \left[ \begin{array}{ccccc}
u_1 & u_1 & \cdots &\cdots & u_1  \\
x_1 & u_2 & \cdots &\cdots & u_2 \\
x_1 & x_2 & u_3 & \cdots & u_3 \\
 \vdots & \vdots & & \ddots & \vdots \\
x_1 & x_2 & \cdots & \cdots & u_n
\end{array} \right] \circ
\left[ \begin{array}{ccccc}
v_1 & v_2 & \cdots & \cdots & v_n  \\
y_2 & v_2 & \cdots & \cdots & v_n \\
y_3 & y_3 & v_3 & \cdots & v_n \\
 \vdots & \vdots & & \ddots & \vdots \\
y_n & y_n & \cdots & \cdots & v_n
\end{array} \right].
\eeq
Thus, the inverses of non-singular tridiagonal matrices are data sparse. For a recent book on tridiagonal matrices  see \cite{Meu25}. 
A frequently asked question about non-singular sparse matrices is if 
there exists a similar structure for their inverses as in the tridiagona, or in other terms, what are generalizations of Green's matrices for arbitrary sparse matrices? 
In general, the answer to this  question is likely to be negative. 

Starting with tridiagonal matrices and their inverses more classes of data sparse matrices were identified. Among them are semi-separable matrices \cite{Van05, Van08} but as we will  see later inverses of acyclic matrices are also data sparse. 

Another  nice property  of non-singular strictly diagonally dominant tridiagonal matrices is that the entries of the inverse decay away  from the diagonal (in absolute values)\cite{Nab99b}.
If $A$ is tridiagonal and strictly  diagonally dominant by rows, then  the sequence of the $\{|u_i|\}$ is strictly
increasing while  the sequence of the $\{|y_i|\}$ is strictly decreasing. However,
if $A$ is tridiagonal and strictly  diagonally dominant by columns, then  the sequence of the $\{|x_i|\}$ is strictly
increasing while the sequence  of the $\{|v_i|\}$ is strictly decreasing  (Theorem 3.2 in \cite{Nab99b}). 

Matrices that are characterized by off-diagonal decay, or more generally ``localization'' of their {entries}, appear in applications throughout the mathematical and computational sciences. The presence of such localization can lead to  computational savings, since it allows to (closely) approximate a given matrix by using its significant entries only, and discarding the negligible ones according to a pre-established criterion. In this context, it is then of great practical interest to know a priori how many and which of these entries can be discarded as insignificant. Therefore, many authors have studied decay rates for different matrix classes and functions of matrices \cite{EchLN18}; see, e.g.,~\cite{BenBoi14,BenGol99,BenRaz07,BenSim15,CanSimVer14,DemMosSmi84,KriStrWer15,PelPol01}. For an excellent survey of the current state-of-the-art we refer to~\cite{Ben16}.

Here we consider irreducible non-singular acyclic matrices, i.e. matrices whose undirected graph is a tree. Note  that the graph  of a tridiagonal matrix is simply a line. Hence, the class of acyclic matrices includes  the class  of tridiagonal matrices. 

Acyclic (or treediagonal) matrices are studied in the landmarked paper by Klein \cite{Kle82}. He showed that the inverse satisfies the so-called triangle property and thus can also be described by just a few parameters. In \cite{Nab01} the author  gives   another equivalent formulation   of the inverses of these matrices. 
The beauty of this result is that it accurately describes the structure of inverses of acyclic matrices in a way that it is  done  for  inverses of tridiagonal matrices. It is proven  
that the inverses of irreducible acyclic symmetric matrices are given 
as the Hadamard product of three matrices, a type D matrix, a flipped type
D matrix (as in (\ref{inva11})) and a matrix of tree structure which is closely related 
to the graph of $A$ itself. A similar result holds for non-symmetric matrices.

Acylic matrices are  studied by many authors. However, often eigenvalue problems or   combinatorial properties of these matrices are in the main focus, see e.g. \cite{ Bru86, DemG93, Fied75, Kim09}, but in \cite{Bri04}  Moore--Penrose inverses of acyclic matrices are considered. Recently in \cite{PraS26} acyclic (treediagonal) matrices are characterized whose inverses are Z-matrices.

In this paper,    we consider a possible decay in the inverses of strictly diagonally dominant irreducible acyclic matrices. Following e.g.\cite{BenRaz07}, an $ n \times n$ matrix $C = [c_{i,j}]$ has the exponential off-diagonal decay property if there is a constant $k > 0$ such that
\beq \label{eq:ex}
 |c_{i,j}| \leq k \lambda^{|i-j|},  \quad \mbox{where}  \quad 0 < \lambda < 1. 
\eeq
Note that here the bounds for  the entries exponential decay if   $|i-j|$ gets larger. 
In this paper  we show, that similar to the inverses of diagonally dominant tridiagonal matrices,  the entries of inverses of acyclic matrices decay, not only the bounds.  
But this decay is somehow hidden and cannot visible immediately. Here we identify this  decay. It is not a 
decay strictly  away from the diagonal. The entries of the inverse  decay  away from the diagonal but  following the path from  a diagonal entry $c_{ii}$ to an $c_{ij}$. 

Moreover  we give lower and upper  bounds for the absolute values of the entries of the inverses of non-singular acyclic strictly diagonally dominant matrices. With  these bounds we can also establish an exponential decay as in (\ref{eq:ex}). The exponent is then $dist(i,j)$, the length of the unique path between  the vertices $i$ and $j$ in the graph of $A$. Exponential decay as in (\ref{eq:ex}) using $dist(i,j)$ is already considered in \cite{BenRaz07} for inverses of sparse diagonalizable and banded matrices.  

Our decay results as well as  our bounds generalize the known result  for inverses of tridiagonal matrices.  Indeed  for tridiagonal matrices our bounds are the same as those given e.g. in \cite{Nab99a, Nab99b}.

\section{Notations and preliminary results}

First   we need some notations and definitions. Here  we follow \cite{Ha72}. A {\it  weighted  graph
G = (V,E)} of $n+1$ vertices is a graph with the vertex set $V = \{0, \ldots, n\}$, 
edges $e_{ij} \in E$ between the 
vertices ($i,j \in \{0,1,\ldots,n\}$) labeled by nonzero weights $w_{i,j} \in \bR$.


A {\it path} 
from vertex $i$ to vertex $j$, denoted by $P_{i,j}$, is the shortest ordered sequence of edges 
\[
P_{i,j} : \left( (k_1,k_2), (k_2, k_3), \ldots (k_{r-1},k_r)\right) \]
where $k_1 = i$ and
$k_r = j$. The path $P_{i,j}$ is a cycle , if $k_1 = i = j = k_r$, $r \geq 3$ and $k_1, \ldots, k_{r-1}$ are distinct. 
A graph is acyclic if it has no cycles. A {\it tree} is a connected acyclic 
graph. Equivalently, a tree is a graph for which there exists a
unique path between  any two vertices $i$ and $j$. 
Here   we consider trees. For the unique  $P_{i,j}$ we also need
\[
\tilde P_{i,j} := \{k_1, k_2, \ldots, k_{r-1}\}.
\]
$\tilde P_{i,j}$ is just  the set of vertices of the path, except the last one. The pathlenght or distance between   the two vertices $i$ and $j$ is the number of edges in  $P_{i,j}$ 
\[
 |\tilde P_{i,j}| = dist(i,j) = r-1.
\]


We  say for short $l$ follows the path  between vertices $i$ and $j$ in a tree, which means we consider the unique sequence of ordered vertices
\beqo
(i,k_2), (k_2,k_3), \ldots, (k_{r-1},j). 
\eeqo
A {\it rooted tree}  is a tree with a prominent vertex called the root.
Here we distinguish between trees $\Gamma $ and rooted trees $\Gamma_{(0)}$. 
We always assume that trees have vertex set $\{1,\ldots,n\}$ while for 
rooted trees $\{0,1,\ldots, n\}$ the vertex set and the root is labeled by $0$. 
For a given tree $\Gamma$ we 
construct a (special) rooted tree $\Gamma_{(0)}$ by adding a new vertex $0$, the root,
and a new edge $e_{0,1}$ from the root to vertex 1 to the old tree.

We number the vertices of the trees in the following way. 
We start with an 
arbitrary vertex which is then labeled with $1$. 
Following we proceed 
recursively by a depth-first 
search or numbering (dfs), see \cite{AHU}.

A {\it branch} starting at vertex $i$ of the  tree $\Gamma$  is the connected subgraph of $\Gamma$ including vertex $i$ obtained by deleting the unique edge $e_{j,i}$ with $ j < i $. Vertex $i$ then becomes the root of this particular branch. Thus  we label the vertices recursively branch by branch.  

Moreover,  we want to make use of so-called levels in a tree $\Gamma_0$. All  vertices that  are connected with  the root build level $1$. Vertices  that are connected with level $1$ vertices, which are not the  root, build level $2$. Inductively, vertices in level $k$ are vertices that  are connected with level $k-1$ vertices except level $k-2$ vertices.  

Then a predecessor of a level $k$ vertex $i$ is the vertex of level $k-1$ connected with vertex $i$. It is denoted  by $p(i)$. Successors of a level $k$ vertex $i$ are those vertices of level $k+1$ which are connected with $i$. These vertices are denoted by $s_j(i)$.  Leafs of the tree are vertices which have  no successors.
Note that  for trees every vertex except the root has exactly one predecessor. Every vertex, except  the leaves, has one or more successors.  


The {\it graph} $G(A)$ 
of an $n \times n$ matrix $A = [a_{i,j}]$ is the undirected graph
consisting of $n$ vertices $\{1,\ldots,n\}$ so that there is an edge between
vertex $i$ and vertex $j$ if and only if $a_{i,j} \neq 0 $ or $a_{j,i} \neq 0 $. 
Here, we consider irreducible acyclic matrices. Thus, we do have a symmetric nonzero pattern in $A$. 

$A$ is called {\it treediagonal} by Klein \cite{Kle82} if $G(A)$ is 
a forest, i.e. a collection of trees. However, in this paper we prefer the name 
{\it acyclic matrices} for matrices $A$ whose graphs  
$G(A)$ are forests.  We consider irreducible acyclic matrices $A$, thus $G(A)$ is a tree. 


If $G(A)$  is a tree then  we consider  the path between  two distinct indices $i$ and $j$, i.e. the ordered sequence  of edges 
$(k_1,k_2), (k_2, k_3),$ $\ldots (k_{r-1},k_r)$, where $k_1 = i$ and
$k_r = j$. We then consider  the sequence of entries of $A$
\beqo
a_{i,i} = a_{k_1,k_1}, a_{k_1, k_2}, a_{k_2, k_3}, \ldots a_{k_{r-1},k_r}, a_{k_{r},k_r} = a_{j,j}.
\eeqo
In sum, we  say for this ordered sequence: the entries  $a_{i,l}$ of $A$, where $l$ follows  the path from $a_{i,i}$ to $a_{j,j}$.

Klein  established in \cite{Kle82} some results on matrices 
whose graphs are trees. We will use these results in the next section.  \\

\begin{definition}
An $n \times n$ matrix $C = [c_{i,j}]$ satisfies the  {\it treeangle property} with respect to  a given tree $\Gamma$ if for every $i,j,k \in V$ with $P_{i,k} \subseteq P_{i,j}$
\beq \label{treeangle}
 c_{i,j}c_{k,k} = c_{i,k}c_{k,j}.
\eeq
\end{definition}

Klein then proved the following theorems 

\begin{theorem} \label{kl1}
Let $\Gamma$ be a tree. If a non-singular matrix $C$ satisfies the treeangle
property with respect to $\Gamma$ and if $c_{i,i} \neq 0$ for all interior
vertices, then $C^{-1}$ is treediagonal with respect to $\Gamma$, i.e. 
$C^{-1}$ is acyclic.
\end{theorem}

\begin{theorem} \label{kl2}
Let $\Gamma$ be a tree. If $A$ is a non-singular treediagonal matrix with respect to $\Gamma$ then $A^{-1}$ satisfies the  treeangle property with respect to 
$\Gamma$.
\end{theorem}

Moreover we have (see \cite{Nab01})

\begin{proposition} \label{Nab01}
Let $A =[a_{i,j}] \in \bRn $ be non-singular, irreducible and acyclic. Assume
that the diagonal entries of $C= [c_{i,j}] := A^{-1}$ are nonzero. Then 
$c_{i,j} \neq 0$ for all $i, j \in \{1,\ldots, n \}$. Hence, if $A$ is  irreducible, acyclic and diagonally dominant, then all entries of $C^{-1}$ are different from zero.  
\end{proposition}

The above Theorems  describe the structure of inverses of acyclic 
matrices. However, it is not clear how one can describe 
the treeangle property in terms of matrices as in the tridiagonal case. Another question that has to beraised is what are the generalizations of Green's matrices?  
To find an answer we need the following class of matrices.

\begin{definition}
Let $\Gamma_{(0)}$ be a weighted rooted tree with vertex set $\{0,1,\ldots, n\}$, where $0$ denotes the root. A matrix $A = [a_{i,j}] \in \bRn$ is of 
{\it tree structure} with respect to $\Gamma_{(0)}$, if for all $i,j = 1,\ldots,n$
\beq \label{dis}
a_{i,j} & = & \sum_{\{r,s\} \in P_{i,0} \cap P_{j,0}} w_{r,s},  
\eeq
here $\{r,s\} \in P_{i,0} \cap P_{j,0}$ denotes a common edge of the paths 
$P_{i,0}$ and  $P_{j,0}$  and $w_{r,s}$ is its weight.
\end{definition}

Note that (\ref{dis}) defines a 'distance' or better an 'inverse distance' 
between the vertices of the tree. This distance was already used  
by Nabben and Varga in \cite{NabV95b} for leaves of trees. \\
If all   weights  are non-negative, matrices of tree structure are ultrametric matrices \cite{NabV94, VarN93, Fie00}. We can also define non-symmetric matrices of tree structure by introducing a second set of weights which are used to build the $a_{j,i}$ similar similar to  (\ref{dis}). This then leads to generalized ultrametric matrices \cite{NabV95a, NabV95b, Nab97, McDNNST98, ElsNN98, DelMM14}. Example \ref{ex1} gives a matrix of tree-structure. The paper \cite{Nab26b} gives several relations between matrices of tree structure, generalized ultrametric matrices and inverse acyclic matrices. 

Moreover, matrices of tree  structure as well as ultrametric matrices can be decomposed as the sum of $2n-1$ rank one matrices, see Theorem 2.5 in \cite{Nab01}. 
\begin{theorem} \label{cha}
Let $\Gamma_{(0)}$ be a weighted rooted tree. The following are equivalent:
\begin{enumerate}
\item[(1)]
$A = [a_{i,j}] \in \bRn$ is
of tree structure with respect to $\Gamma_{(0)}$.
\item[(2)]
$A$ can be decomposed as
\beq \label{decom}
A = \sum_{i=1}^n \tau_i u_iu_i^T
\eeq
where $\tau_i \in \bR$ and $\tau_i = w_{j,i}$, here $j$ is the unique vertex
which is connected with $i$ and $j < i$. The vectors $u_i \in \bRn$ satisfy
$(u_i)_j =1 $ for all $j$ (including $i$) belonging to the branch of $\Gamma_{(0)}$ starting at vertex $i$, and $(u_i)_j = 0$ otherwise.
\end{enumerate}

\end{theorem}

We  then can state the main result of \cite{Nab01}, which describes the   structure of the inverse of an acyclic matrix. These results  generalize  the classical results  for tridiagonal matrices. 


\begin{theorem} \label{main-structure}
Let $A \in \bRn$ be non-singular, irreducible and acyclic and let $\Gamma = G(A)$.  Assume that the diagonal entries of $A^{-1}$ are nonzero. 
Then there exist matrices  $T$ and $R$ of the form 

\beq \label{def-neu}
 T = \left[ \begin{array}{ccccc}
d_1 & d_1 & \cdots &\cdots & d_1  \\
f_1 & d_2 & \cdots &\cdots & d_2 \\
f_1 & f_2 & d_3 & \cdots & d_3 \\
 \vdots & \vdots & & \ddots & \vdots \\
f_1 & f_2 & \cdots & \cdots & d_n
\end{array} \right] \quad \quad
R = \left[ \begin{array}{ccccc}
f_1 & f_2 & \cdots & \cdots & f_n  \\
d_2 & f_2 & \cdots & \cdots & f_n \\
d_3 & d_3 & f_3 & \cdots & f_n \\
 \vdots & \vdots & & \ddots & \vdots \\
d_n & d_n & \cdots & \cdots & f_n
\end{array} \right],
\eeq
with $f_1 = d_1$, and a matrix $U$ of tree structure with respect to the weighted tree $\Gamma_{(0)}$ such that

\beqo
A^{-1} = T \circ U \circ R.
\eeqo
Conversely, let $U = [u_{i,j}] \in \bRn$ be non-singular. If $U$ is of tree structure with respect to a given rooted
tree $\Gamma_{(0)}$, then for any matrices $T$ and $R$ of the form 
$(\ref{def-neu})$ with $f_1 = d_1$ and  $f_id_i \neq 0$ for all $i$, the matrix 
$(T \circ U \circ R)^{-1}$ is irreducible, acyclic and $G((T \circ U \circ R)^{-1}) = \Gamma.$
\end{theorem}

Note again that here $\Gamma_{(0)} $ is obtained from $\Gamma$ by adding the root $0$ and a 
new edge $e_{0,1}$.

The sequences $\{f_i\}$ and $\{d_i\}$ are given by 
$F = diag(f_1,\ldots,f_n) = diag(e_1^TA^{-1})$ and 
$D = diag(d_1,\ldots,d_n) = diag(A^{-1}e_1)$.

We than have $3n - 2$ parameters in Theorem \ref{main-structure}, which gives the inverse of an acyclic matrix. The $n-1$ parameters for the matrix of tree  structure, and $2n -1$ for the $f_j$ and $d_j$, since $f_1 = d_1$.  So  we have exactly the same number of parameters for the acyclic matrix as for its inverse. 

For the symmetric case we obtain

\begin{corollary} \label{coro-hadasym}
Let $A \in \bRn$ be symmetric non-singular, acyclic and irreducible with  
$G(A) = \Gamma$. Assume that the diagonal entries of $A^{-1}$ are nonzero. 
Then there exist matrices  $T$ and $R$ of the form 

\beq \label{defsym-neu}
 T = \left[ \begin{array}{ccccc}
d_1 & d_1 & \cdots &\cdots & d_1  \\
d_1 & d_2 & \cdots &\cdots & d_2 \\
d_1 & d_2 & d_3 & \cdots & d_3 \\
 \vdots & \vdots & & \ddots & \vdots \\
d_1 & d_2 & \cdots & \cdots & d_n
\end{array} \right] \quad \quad
R = \left[ \begin{array}{ccccc}
d_1 & d_2 & \cdots & \cdots & d_n  \\
d_2 & d_2 & \cdots & \cdots & d_n \\
d_3 & d_3 & d_3 & \cdots & d_n \\
 \vdots & \vdots & & \ddots & \vdots \\|c_{i,i}| \epsilon_{ij}^{r-1} 
d_n & d_n & \cdots & \cdots & d_n
\end{array} \right].
\eeq
and a matrix $U$ of tree structure with respect to the weighted tree $\Gamma_{(0)}$ such that

\beq \label{hada-sym}
A^{-1} = T \circ U \circ R.
\eeq
Conversely, let $U = [u_{i,j}] \in \bRn$ be non-singular. If $U$ is of tree structure with respect to a given rooted
tree $\Gamma_{(0)}$. Then for any matrices $T$ and $R$ of the form 
$(\ref{defsym})$ with $d_i \neq 0$ for all $i$, the matrix 
$(T \circ U \circ R)^{-1}$ is irreducible, acyclic and $G((T \circ U \circ R)^{-1}) = \Gamma.$
\end{corollary}


\begin{example} \label{ex1}
{\rm For illustration consider the weighted rooted tree in Figure 1. The 
$6 \times 6$ matrix of tree 
structure is then given by }
\[
A = \left[ \begin{array}{rrrrrr}
-1 & -1 & -1 & -1 & -1 & -1\\
-1 & 2 & 2 & -1 & -1 & -1 \\
-1 & 2 & 3 & -1 & -1 & -1 \\
-1 & -1 & -1 & 1 & 1 & 1 \\
-1 & -1 & -1 &  1 & -1 & 1 \\
-1 & -1 & -1 & 1 & 1 & 2
\end{array}\right] .
\]
{\rm Moreover $A$ can be decomposed as }

\beqo
A = \sum_{i =1}^{5} ~ \tau_{i} {\bf u}_{i} {\bf u}_{i}^{T}
\eeqo
{\rm where the $u_i$ are given as in Figure 1 and the  
$\{ \tau_{1},\ldots, \tau_{6}\} = \{ -1,3,1,2,-2,1\}$. }

{\rm 
\noindent
\unitlength0.8cm
\begin{picture}(12,7)
\put(8.5,0){\makebox(0,0)[lb]{$u_6:=[0,0,0,0,0,1]^T$}}
\put(8.5,0.7){\makebox(0,0)[lb]{$u_3:=[0,0,1,0,0,0]^T$; $u_5:=[0,0,0,0,1,0]^T$}}
\put(8.5,2){\makebox(0,0)[lb]{$u_2:=[0,1,1,0,0,0]^T$; $u_4:=[0,0,0,1,1,1]^T$}}
\put(8.5,4){\makebox(0,0)[lb]{$u_1:=[1,1,1,1,1,1]^T$}}
\put(5.5,0){\circle*{0.2}}
\put(7.5,0){\circle*{0.2}}
\put(3.5,0){\circle*{0.2}}
\put(4.5,2){\circle*{0.2}}
\put(6.5,2){\circle*{0.2}}
\put(5.5,4){\circle*{0.2}}
\put(4,6){\circle*{0.2}}
\put(5.8,0){\makebox(0,0)[lb]{{\bf 5}}}
\put(5.3,1){\makebox(0,0)[lb]{-2}}
\put(4.2,0.7){\makebox(0,0)[lb]{1}}
\put(7.3,1){\makebox(0,0)[lb]{1}}
\put(4.3,3){\makebox(0,0)[lb]{3}}
\put(6.3,3){\makebox(0,0)[lb]{2}}
\put(5.0,5){\makebox(0,0)[lb]{-1}}
\put(3.8,0){\makebox(0,0)[lb]{{\bf 3}}}
\put(4.7,2){\makebox(0,0)[lb]{{\bf 2}}}
\put(7.8,0){\makebox(0,0)[lb]{{\bf 6}}}
\put(6.8,2){\makebox(0,0)[lb]{{\bf 4}}}
\put(5.8,4){\makebox(0,0)[lb]{{\bf 1}}}
\put(4.4,6){\makebox(0,0)[lb]{{\bf 0}}}
\put(5.5,0){\line(1,2){1}}
\put(7.5,0){\line(-1,2){1}}
\put(3.5,0){\line(1,2){1}}
\put(4.5,2){\line(1,2){1}}
\put(6.5,2){\line(-1,2){1}}
\put(5.5,4){\line(-3,4){1.5}}
\end{picture}
\vspace{.05in}

\begin{center}
{\sc Figure 1} 
\end{center}
}
\end{example}


%






Ostrowski established in \cite{O2} upper and lower bounds for the entries of the inverse of an arbitrary diagonally dominant matrix. 


\begin{lemma} \label{wc}
Let $A$ be a strictly row diagonally dominant, then $A^{-1} = C = [c_{i,j}]$ exists  and for
$i \neq j$ we have 
\beq \label{os}
|c_{i,j}| \leq \mu_i|c_{i,i}| \leq |c_{i,i}|.
\eeq
and 
\beq \label{osab} 
\frac{1}{|a_{i,i}|(1 + \mu_i)} \leq |c_{i,i}| \leq \frac{1}{|a_{i,i}|(1 - \mu_i)},
\eeq
where 
\beq \label{muus}
\mu_i := \frac{1}{|a_{i,i}|}\sum_{j\neq i}|a_{i,j}|.
\eeq

\end{lemma}

Note, that a similar  result  holds  for strictly column diagonally dominant matrices.

\section{Main Results}

We  start this  section  with some bounds for the entries of the inverses of acyclic matrices. First we need  some notations.

 Let $A_{[i,j]}$ be the submatrix of $A$ which is obtained by deleting the $i$-th  row and $j$-th column of $A$.  Let $A^{(i,k)}$, for $i < k$, be the principal submatrix of $A$ which is built by the $i, i+1, \ldots,  k$ rows and columns of $A$.  \\

\begin{lemma} \label{lem.det1}
 Let $A = [a_{i,j}] \in \bRn$ be an acyclic and irreducible matrix  with   
$G(A) = \Gamma$. For a fixed but arbitrary $i \in \{1, \ldots, n\}$  consider the predecessor $p(i)$  of $i > 1$. Let $\widetilde A$ be the matrix obtained from $A$ by permuting the rows and columns of $A$  such that the $i-1$-th row and column of  $\widetilde A$  is the $p(i)$-th row and column of  $A$.
For a fixed successor $s_l(i)$ of $i$ (if possible) let $\widehat A$ be the matrix obtained from $A$ by permuting the rows and columns of $A$  such that the $(i+1)$-th row and column of  $\widehat A$  is the $s_l(i)$-th row and column of  $A$.

Then
\beq \label{det-streicha}
det(\widetilde A_{[i,i-1]}) = (-1)^{i+i-1} \widetilde a_{i-1,i} det(\widetilde A^{(1,i-2})det(\widetilde A^{(i+1,n)}), 
\eeq
\beq \label{det-streich-4a}
det(\widehat  A_{[i+1,i]}) = (-1)^{i+i+1}  \widehat a_{i,i+1} det(\widehat A^{(1,i-1}) det(\widehat A^{(i+2,n)}).
\eeq

Moreover,
\beq \label{det-streich-6}
det(\widetilde  A_{[i,i]}) = det(\widetilde A^{(1,i-1}) det(\widetilde A^{(i+1,n)}),  
\eeq

\beq \label{det-streich-666}
det(\widehat  A_{[i,i]}) = det(\widehat A^{(1,i-1}) det(\widehat A^{(i+1,n)}).
\eeq
\end{lemma}
\begin{proof}
The cases $i = 2$ and $i = n-1$ are obvious. 
Next consider $(\ref{det-streicha})$.
 Then  we have
 \beqo
 \widetilde A_{[i,i-1]} = \left[ \begin{array}{ccc}
\widetilde A^{(1,i-2)} & 0 & 0 \\
x & \widetilde a_{(i-1,i)} & y \\
x  & x & \widetilde A^{(i+1,n)} 
\end{array} \right]. 
\eeqo
We have the upper zero block since there are no connections from a vertex $j \in \{1, \ldots i-2 \}$ to vertices in   $\{i-1, \ldots n\}$  due to the numbering of the vertices. In $y$ there might be entries which are different from zero, but only if vertex $p(i)$ has a successors different from $i$. Now, if we expand 
\beqo
 det \left(\left[ \begin{array}{cc}
\widetilde a_{(i-1,i)} & y \\
x & \widetilde A^{(i+1,n)} 
\end{array} \right] \right). 
\eeqo
along the first row, the entries different from zero in $y$ will  be multiplied with a determinant of a matrix which has a zero column.

Next  consider $(\ref{det-streich-4a})$. We have
 \beqo
 \widehat A_{[i+1,i]} = \left[ \begin{array}{ccc}
\widehat A^{(1,i-1)} & 0 & 0 \\
a & \widehat a_{(i,i+1)} & y \\
b  & c & \widehat A^{(i+2,n)} 
\end{array} \right]. 
\eeqo
Again the upper zero block appears since  there are no connections from a vertex $j \in \{1, \ldots i-1 \}$ to vertices in   $\{i+2, \ldots n\}$  due to the numbering of the vertices. In $c$ there might be entries which are different from zero, but only if vertex $s_l(i)$ has successors. Now, if we expand 
 \beqo
det \left( \left[ \begin{array}{cc}
\widehat a_{(i,i+1)} & y \\
c & \widehat A^{(i+2,n)} 
\end{array} \right] \right) 
\eeqo
along the first column, these entries in $c$ different from zero are  multiplied with a determinant of a matrix which has a zero column.
\end{proof}

We then obtain one of our main  results. 

\begin{theorem} \label{theo:det1}
 Let $A = [a_{i,j}] \in \bRn$ be an acyclic and irreducible strictly row and column diagonally dominant  matrix with   
$G(A) = \Gamma$ and   $A^{-1} = C = [c_{i,j}]$. For all $i \in \{1, \ldots, n\}$  consider the predecessor $p(i)$  of $i > 1$  and one fixed successor $s(i)$  of $i$ (if possible). 
Then
\beq \label{det-streich17a}
\frac{|c_{p(i),i}|}{ |c_{i,i}| } \geq \frac{|a_{p(i),i}|} { |a_{p(i),p(p(i))}| +  |a_{p(i),p(i)}| +   \sum_{l \in S_{p(i), l \neq i, l < p(i)}}  |a_{p(i),l}|} =: \rho_{p(i),i},
\eeq

\beq \label{det-streich17ad}
\frac{|c_{p(i),i}|}{ |c_{i,i}| } \leq \frac{|a_{p(i),i}|} { |a_{p(i),p(p(i))}| -  |a_{p(i),p(i)}| -   \sum_{l \in S_{p(i), l \neq i, l < p(i)}}  |a_{p(i),l}|} =: \tau_{p(i),i}.
\eeq
 
Moreover
\beq \label{det-streich17bc}
\frac{|c_{i,s(i)}|}{ |c_{i,i}| } \geq \frac{|a_{i,s_{(i)}}|} { |a_{s(i),s(i)}| +   \sum_{l \in S_{i}, l \neq s(i)}  |a_{l,s(i)}|} =: \sigma_{i,s(i)}
\eeq
and 
\beq \label{det-streich17bcd}
\frac{|c_{i,s(i)}|}{ |c_{i,i}| } \leq \frac{|a_{i,s_{(i)}}|} { |a_{s(i),s(i)}|  -   \sum_{l \in S_{i}, l \neq s(i)}  |a_{l,s(i)}|} =: \omega_{i,s(i)}.
\eeq
\end{theorem}

\begin{proof}
 First consider $(\ref{det-streich17a})$. Let $\widetilde A$ be the matrix obtained from $A$ by permuting the rows and columns of $A$  such that the $(i-1)$-th row and column of  $\widetilde A$  is the $p(i)$-th row and column of  $A$.
 
 Note that due  to the numbering of the vertices we have  $p(i) < i$.  Moreover
\beqo 
\frac{|c_{p(i),i}|}{ |c_{i,i}| } &  =  & 
 \frac{|det(A_{[i,p(i)]})|}{ |det(A_{[i,i]})|} \\
 & = & 
|a_{p(i),i}|
\frac{|det(\widetilde A^{(1,i-2})|*|det(\widetilde A^{(i+1,n})| }{|det(\widetilde A^{(1,i-1)})|*|det(\widetilde A^{(i+1,n})|} \\
 & = & 
|a_{p(i),i}|
\frac{|det(\widetilde A^{(1,i-2})|}{|det(\widetilde A^{(1,i-1)})|} \\
& = &  |a_{p(i),i}| |(inv(\widetilde A^{(1,i-1)}))_{i-1,i-1}|.
\eeqo
But $ \widetilde A^{(1,i-1)})$ is a $p(i) \times p(i)$ strictly diagonally dominant matrix, thus using Lemma \ref{wc}, and using the original matrix $A$ again, we obtain for the last  row of $\widetilde A^{(1,i-1)}$ 
\beqo
|(inv(\widetilde A^{(1,i-1)}))_{i-1,i-1}| \geq 
\frac{1}{|a_{p(i),p(p(i))}| +  |a_{p(i),p(i)}| +   \sum_{l \in S_{p(i), l \leq p(i)}}  |a_{p(i),l}|} 
\eeqo
and
\beqo
|(inv(\widetilde A^{(1,i-1)}))_{i-1,i-1}| \leq 
\frac{1}{|a_{p(i),p(p(i))}| -  |a_{p(i),p(i)}| -  \sum_{l \in S_{p(i), l \leq p(i)}}  |a_{p(i),l}|}.
\eeqo
Thus
\beqo
\frac{|c_{p(i),i}|}{ |c_{i,i}| } \geq \frac{|a_{p(i),i}|} { |a_{p(i),p(p(i))}| +  |a_{p(i),p(i)}| +   \sum_{l \in S_{p(i), l \neq i, l < p(i)}}  |a_{p(i),l}|} =: \rho_{p(i),i} 
\eeqo
and 
\beqo
\frac{|c_{p(i),i}|}{ |c_{i,i}| } \leq \frac{|a_{p(i),i}|} { |a_{p(i),p(p(i))}| -  |a_{p(i),p(i)}| -   \sum_{l \in S_{p(i), l \neq i, l < p(i)}}  |a_{p(i),l}|} =: \tau_{p(i),i}. 
\eeqo

 Next consider (\ref{det-streich17bc}). Let $\widehat A$ be the matrix obtained from $A$ by permuting the rows and columns of $A$  such that the $i+1$-th row and column of  $\widehat A$  is the $s(i)$ -th row and column of  $A$.
 Here we choose one particular successor of $i$ and denote it by $s(i)$. Note, that  there might be more than one. We then  have
\beqo
\frac{|c_{i,s(i)}|}{ |c_{i,i}| } &  =  & 
 \frac{|det(A_{[s(i),i]})|}{ |det(A_{[i,i]})|} \\
 & = & 
|a_{i,s(i)}|
\frac{|det(\widehat A^{(1,i -1})|*|det(\widehat A^{(i+2,n)})|}
{|det(\widehat A^{(1,i -1})|   *  |det(\widehat A^{(i+1,n)})|}                    \\
& = & |a_{i,s(i)}|
\frac{|det(\widehat A^{(i+2,n)})|}{|det(\widehat A^{(i+1,n)})|}\\
& = &  |a_{i,s(i)}| |(inv(\widehat A^{(i + 1,n)}))_{1,1}|
\eeqo
But $ \widehat A^{(i+1,n)})$ is a strictly diagonally dominant matrix, thus using Ostrowski's result, i.e. Lemma \ref{wc},   and using the original matrix $A$ again, we obatin for the first   row of $\widehat A^{(i+1,n)}$ 
\beqo
|(inv(\widehat A^{(i+1,n)}))_{1,1}| \geq 
\frac{1}{|a_{i,s(i)}| +  |a_{s(i),s(i)}| +   \sum_{l \in S_{i}, l \neq s(i) }|a_{l,s(i)}|}
\eeqo
and
\beqo
|(inv(\widehat A^{(i+1,n)}))_{1,1}| \leq 
\frac{1}{|a_{i,s(i)}| -  |a_{s(i),s(i)}| -  \sum_{l \in S_{i}, l \neq s(i) }|a_{l,s(i)}|}
\eeqo

Thus
\beqo
\frac{|c_{i,s(i)}|}{ |c_{i,i}| } \geq \frac{|a_{i,s(i)}|} { |a_{s(i),s(i)}| +   \sum_{l \in S_{i}, l \neq s(i)}  |a_{l,s(i)}|} =: \sigma_{i,s(i)}
\eeqo
and
\beqo
\frac{|c_{i,s(i)}|}{ |c_{i,i}| } \leq \frac{|a_{i,s(i)}|} { |a_{s(i),s(i)}| -   \sum_{l \in S_{i}, l \neq s(i)}  |a_{l,s(i)}|} =: \omega_{i,s(i)}, 
\eeqo

which  completes the proof. 
\end{proof}

Similarly we obtain the next Theorem. 

\begin{theorem} \label{theo:det2}
 Let $A = [a_{i,j}] \in \bRn$ be an acyclic and irreducible strictly row and column diagonally dominant  matrix with   
$G(A) = \Gamma$ and   $A^{-1} = C = [c_{i,j}]$. For all $i \in \{1, \ldots, n\}$  consider the predecessor $p(i)$  of $i > 1$  and one fixed successor $s(i)$  of $i$ (if possible). 
Then
\beq \label{det-streich19a}
\frac{|c_{i,p(i)}|}{ |c_{i,i}| } \geq \frac{|a_{i,p(i)}|} { |a_{p(p(i)),p(i)}| +  |a_{p(i),p(i)}| +   \sum_{l \in S_{p(i), l < i}}  |a_{l,p(i)}|} =: \widetilde \rho_{i,p(i)},
\eeq

\beq \label{det-streich19ad}
\frac{|c_{i,p(i)}|}{ |c_{i,i}| } \leq \frac{|a_{i,p(i)}|} { |a_{p(p(i)),p(i)}| -  |a_{p(i),p(i)}| -   \sum_{l \in S_{p(i),  l < i}}  |a_{l,p(i)}|} =: \widetilde \tau_{i,p(i)}.
\eeq
 
Moreover
\beq \label{det-streich19bc}
\frac{|c_{s(i),i}|}{ |c_{i,i}| } \geq \frac{|a_{s_{(i),i}}|} { |a_{s(i),s(i)}| +   \sum_{l \in S_{i}, l \neq s(i)}  |a_{s(i),l}|} =: \widetilde \sigma_{s(i),i}
\eeq
and 
\beq \label{det-streich19bcd}
\frac{|c_{s(i),i}|}{ |c_{i,i}| } \leq \frac{|a_{s_{(i),i}}|} { |a_{s(i),s(i)}|  -   \sum_{l \in S_{i}, l \neq s(i)}  |a_{s(i),l}|} =: \widetilde \omega_{s(i),i}.
\eeq
\end{theorem}
\begin{proof}
 The statements follow from Theorem \ref{theo:det1} by considering $A^T$. 
\end{proof}

Similar bounds as above  can be obtained by using a result  proved by Klein \cite{Kle82}.  Let $\Delta_{[i,k]}$, $ i \leq k$,  be the  determinant of the principal submatrix of $A$ obtained by deleting the row and columns $i, i+1, \ldots, k$. Then from Theorem 1 in \cite{Kle82} it follows for two vertices $i$ and $k$ which are connected
\beqo
\frac{|c_{i,k}|}{ |c_{i,i}| }  = |a_{i,k}|  \left|\frac{\Delta_{[i,k]}}{\Delta_{[i,i]}}\right|.
\eeqo
But  then  the term   
\beq  \label{eq:frac}
\frac{|\Delta_{[i,k]}|}{|\Delta_{[i,i]|}}
\eeq
has to be considered. Note, if $i$ is connected to $k$, it does not mean that $k = i+1$ or $k = i-1$. Thus, as in our proof  techniques, you have to permute rows and columns of $A$ such that $k = i+1$ or $k = i-1$ holds. Then (\ref{eq:frac}) is again nothing else than the absolute value of a diagonal entry of the  inverse of a permuted submatrix of $A$ which can be bounded  by Lemma \ref{wc}. 
Thus, again,  the numbering  of the vertices of the tree is also a major issue. Hence, similar  to our Theorems above, different cases have to be studied.
Moreover, the inverses we used in the proofs of Theorems \ref{theo:det1} and \ref{theo:det2} are smaller which leads to sharper bounds. \\

 Now, let 
$P_{i,j} =\{(k_1,k_2), (k_2, k_3), \ldots (k_{r-1},k_r)\}$, where $k_1 = i$ and
$k_r = j$, be the unique path from $i$ to $j$. As mentioned in \cite{Kle82} 
the treeangle property implies 
\beq \label{path-n}
c_{i,j} & = & c_{k_1,k_2} \prod_{s=2}^{r-1} \frac{c_{k_s,k_{s+1}}}{c_{k_s,k_s}}.
\eeq
Then we obtain  
\beq \label{eq:path}
c_{i,j} = c_{i,i}\prod_{s=1}^{r-1} \frac{c_{k_s,k_{s+1}}}{c_{k_s,k_s}},
\eeq
but this is an equality. So we have for all indices $k_t$ of the path 
\beqo
c_{i,k_{t}} = c_{i,i}\prod_{s=1}^{t-1} \frac{c_{k_s,k_{s+1}}}{c_{k_s,k_s}}
\eeqo
and 
\beq \label{eq:prod}
c_{i,k_{t+1}} = c_{i,i}\prod_{s=1}^{t} \frac{c_{k_s,k_{s+1}}}{c_{k_s,k_s}} = c_{i,k_t} \frac{c_{k_t,k_{t+1}}}{c_{k_{t}, k_{t}}}.
\eeq

Note, in the product  above, i.e. in the  path from vertex $i$ to vertex $j$ , it is not clear, if vertex $k_{s+1}$  
is a successor or predecessor of vertex $k_s$ in the underlying tree. Here  the indices $k_s$ are just the numbers of the vertices  along the path. Of course   we cannot say if $k_{s+1}$ is less or larger  than $k_{s}$.  However, in Theorem \ref{theo:det1}  we consider   predecessor of a vertex, which have a lower number then the vertex and successors of a vertex, which have  a higher number than  the vertex.  

Therefore,  we define for two indices $t,l$ with $t \neq l$
\beq \label{alpha-ij}
\alpha_{t,l} = \left\{
\begin{array}{ll}
\rho_{t,l} & \quad t < l, \\
\sigma_{t,l} & \quad t > l 
\end{array}
\right.
\eeq
and 
\beq \label{gamma-ij}
\gamma_{t,l} = \left\{
\begin{array}{ll}
\tau_{t,l} & \quad t < l, \\
\omega_{t,l} & \quad t > l. 
\end{array}
\right. 
\eeq



We then can prove  the following Theorem. 

\begin{theorem} \label{theo-bounds-path}
 Let $A = [a_{i,j}] \in \bRn$ be acyclic, irreducible and strictly row and column diagonally dominant matrix with   
$G(A) = \Gamma$ and   $A^{-1} = C = [c_{i,j}]$. Now, let 
$P_{i,j} =\{(k_1,k_2), (k_2, k_3), \ldots (k_{r-1},k_r)\}$, where $k_1 = i$ and
$k_r = j$, be the unique path from $i$ to $j$  with $i \neq j$.
Then
\beq \label{eq:lower}
  |c_{i,i}| \prod_{s=1}^{r-1} \alpha_{k_s,k_{s+1}} \leq |c_{i,j}| \leq  |c_{i,i}| \prod_{s=1}^{r-1} \gamma_{k_s,k_{s+1}}.
\eeq
\end{theorem}
\begin{proof}
 Consider  equation (\ref{eq:path}). Using the  definition (\ref{alpha-ij}) of the $\alpha$ and (\ref{gamma-ij}) of the $\gamma$ then Theorem   \ref{theo:det1} leads to the desired   result.
\end{proof}

The above  theorem  generalizes Theorem 3.3 in \cite{Nab99a}. 
If the graph  is just a line - i.e. for tridiagonal matrices, then for $i < j$ all the $\gamma_{k_s,k_{s+1}}$ in (\ref{eq:lower}) are $ \sigma_{k_s,k_{s+1}}$. While for $i > j$ all the $ \gamma_{k_s,k_{s+1}}$ in (\ref{eq:lower}) are $\rho_{k_s,k_{s+1}}$. The same holds for the $\alpha_{k_s,k_{s+1}}$.

Note using the numbering we have used for the vertices of the tree, the sequences of the $\gamma_{k_s,k_{s+1}}$ switches only once from  some $\rho_{k_l,k_{l+1}}$ to the next $\sigma_{k_{l+1},k_{l+2}}$ or vice versa. The same  holds  for the sequence of the $\alpha_{k_s,k_{s+1}}$.

A similar Theorem as Theorem \ref{theo-bounds-path} can be proven for  the same path, but then the bounds for $|c_{i,j}|$ are related to the diagonal entry $c_{j,j}$ (not $c_{i,i}$) and a product which then  uses
$\tilde \alpha$ and $\tilde \gamma$  defined by
\beq \label{talpha-ij}
\tilde \alpha_{t,l} = \left\{
\begin{array}{ll}
\tilde \rho_{t,l} & \quad t < l, \\
\tilde \sigma_{t,l} & \quad t > l 
\end{array}
\right.
\eeq
and 
\beq \label{tgamma-ij}
\tilde \gamma_{t,l} = \left\{
\begin{array}{ll}
\tilde \tau_{t,l} & \quad t < l, \\
\tilde \gamma_{t,l} & \quad t > l. 
\end{array}
\right. 
\eeq
for  two indices $t,l$ with $t \neq l$. 

Observe that 
 due to the strict row and column diagonal dominance all the 
 $\rho, \tau, \gamma, \sigma$ and  $\tilde \rho, \tilde \sigma, \tilde \tau, \tilde \gamma$ are less than one.



Next we give upper and lower bounds for the diagonal entries.

\begin{theorem} \label{zweib}
 Let $A = [a_{i,j}] \in \bRn$ be acyclic, irreducible and strictly row and column  diagonally dominant.
 Let $A^{-1} = C = [c_{i,j}]$. 
\beqo \label{diabscha}
\frac{1}{|a_{i,i}| +  \tau_{i,p(i)} |a_{i,p(i)}| +  \sum_{s \in S_i} \tilde \omega_{i,s} |a_{i,s}|} 
\leq  |c_{i,i}| \leq \frac{1}{|a_{i,i}| -  \tau_{i,p(i)} |a_{i,p(i)}| -  \sum_{s \in S_i} \tilde \omega_{s,i} |a_{i,s}|}
\eeqo
provided that the denominator of the upper  bound of 
is not zero.
\end{theorem}

\begin{proof}
Since $AA^{-1} = I$ we have for for $i$
\beqo
a_{i,p(i)}c_{p(i),i} + a_{i,i}c_{i,i} + \sum_{s \in S_i}a_{i,s} c_{s,i} = 1.
\eeqo
Thus
\beqo
|a_{i,p(i)}c_{p(i),i}| + \sum_{s \in S_i}|a_{i,s} c_{s,i}| \geq | 1 -  a_{i,i}c_{i,i}|.
\eeqo
Hence 
\beqo
| 1 -  a_{i,i}c_{i,i}| \leq \left( |a_{i,p(i)}|\frac{c_{p(i),i}}{|c_{i,i}|} + \sum_{s \in S_i}|a_{i,s}| \frac{c_{s,i}}{|c_{i,i}|} \right) |c_{i,i}|.
\eeqo
Using Theorem \ref{theo:det1} and  Theorem \ref{theo:det2} we obtain 
\beqo
|1 - a_{i,i}c_{i,i}| \leq (\tau_{i,p(i)} |a_{i,p(i)}| +  \sum_{s \in S_i} \tilde \omega_{s,i} |a_{i,s}|) |c_{i,i}|.
\eeqo
Hence we get the desired bounds.
\end{proof}

Note, in our Theorems above Ostrowski's bound for the diagonal entries (Lemma \ref{wc}) was used to get  bounds for the off-diagonal entries of certain submatrices. The bound is 
\[
|c_{i,i}| \leq \frac{1}{|a_{i,i}|(1 - \mu_i)} = \frac{1}{ |a_{i,i}| - \sum_{j\neq i}|a_{i,j}|}.
\]
But Theorem \ref{zweib} gives 
\[
 |c_{i,i}| \leq \frac{1}{|a_{i,i}| -  \tau_{i,p(i)} |a_{i,p(i)}| -  \sum_{s \in S_i} \tilde \omega_{s,i} |a_{i,s}|}.
\]
Since all $ \tau_{i,p(i)}$ and $\tilde \omega_{s,i} $ are less than one, the bound of Theorem \ref{zweib} is sharper than Ostrowski's bound. Thus one can use this new bound (for the  diagonal entries) to get better  bounds for the off-diagonal entries, see Theorems \ref{theo:det1} and \ref{theo:det2}. But this then leads  to  smaller    
$\tau_{i,p(i)}$ and $\tilde \omega_{s,i}$, i.e.  better bounds for the off-diagonal entries. These new   $\tau_{i,p(i)}$ and $\tilde \omega_{s,i} $  can be used to get sharper bounds for the diagonal entries, Theorem \ref{zweib}. Hence, one can iteratively improve our  bounds as it was introduced for tridiagonal matrices in \cite{Nab99a}. It was observed in \cite{Nab99a} that this refinement improves the bounds significantly, if $A$ is a strictly diagonally dominant $M$-matrix, i.e. a matrix whose diagonal entries are positive and the off-diagonal entries are nonpositive. A similar refinement of bounds is given in \cite{EchLN18} 
for block diagonally dominant  matrices.


Using  (\ref{eq:prod}) and our upper bounds we can identify the   decay in the inverses of acyclic matrices.   

\begin{theorem}\label{theo-main6}
 Let $A = [a_{i,j}] \in \bRn$ be  strictly row and column diagonally dominant  acyclic and irreducible matrix  with   
$G(A) = \Gamma$ and $A^{-1} = C = [c_{i,j}]$.
For each pair $i,j \in \{1, \ldots , n\}$  let 
$P_{i,j} =\{(k_1,k_2), (k_2, k_3), \ldots (k_{r-1},k_r)\}$, where $k_1 = i$ and
$k_r = j$, be the unique path from $i$ to $j$  with $i \neq j$. Then the sequence  $|c_{i,l}|$ where $l$ follows the path from vertex $i$ to vertex $j$ in $G(A)$,  strictly decays. Similarly, the sequence  
 $|c_{l,i}|$ strictly decays. 
 With 
\beq \label{epsi-ij}
\epsilon_{ij} : = \max_{s \in \tilde P_{i,j}} \gamma_{s,s+1},
\eeq
where the $ \gamma_{s,s+1}$ are given as in $(\ref{gamma-ij})$ the $|c_{i,j}|$ can be bounded by
\beq \label{eq:expo}
  |c_{i,j}| \leq  |c_{i,i}| \epsilon_{ij}^{r-1} = |c_{i,i}| \epsilon_{ij}^{dist(i,j)}.
\eeq

\end{theorem}
\begin{proof}
 With (\ref{eq:prod})  we have for every pair of indices $k_t$ and $k_{t+1}$ in $P_{i,j}$
  \beqo 
|c_{i,k_{t+1}}| = |c_{i,i} |\prod_{s=1}^{t} \frac{|c_{k_s,k_{s+1}}|}{|c_{k_s,k_s}|} = |c_{i,k_t}| \frac{|c_{k_t,k_{t+1}}|}{|c_{k_{t}, k_{t}}|}.
\eeqo
But 
\beqo
\frac{|c_{k_t,k_{t+1}}|}{|c_{k_{t}, k_{t}}|}
\eeqo
 can be bounded by Theorem \ref{theo:det1}. Since $A$ is strictly row and column  diagonally dominant,  the bounds are all  less  than one. Equation (\ref{eq:expo}) follows immediately. 
\end{proof}

Thus Theorem \ref{theo-main6} gives a decay along a path in the tree. We do not have a decay strictly away from the diagonal entries. This decay for inverses of acyclic matrices is somehow hidden. 
We need  to have both row and column diagonal dominance to obtain decay along every path. 
Note in (\ref{det-streich17a}) the denominator is just the column sum of the $p(i)$-th row of $A^{(1,p(i))}$. But in (\ref{det-streich17bc}) this is the row sum of the $s(i)$-th column of $A^{(s(i),n)}$.
We need  to do so,  since we want to couple the numerators and denominators.  But here  one can see, that  the bounds and the decay only holds  for {\bf all} paths, if we assume  that  $A$ is row and column diagonally dominant, see Section \ref{sec:examples}. If the diagonal dominance gets larger than the decay becomes  stronger.

Next we will  generalizes another decay result  for  tridiagonal matrices.
We need  the following proposition.

\begin{proposition} \label{propo-one}
Let $A \in \bRn $ be symmetric and of  tree  structure. Then
\[
 a_{1,2} = a_{1,3} = \ldots = a_{1,n} = a_{2,1} = \ldots = a_{n,1}. 
\]
If the weights of the underlying tree of $A$ are non-negative, then the entries $a_{i,l}$, where $l$ follows the path  between $i$ and $j$ decay. 
\end{proposition}
\begin{proof}
 The statement follows  from the definition of a matrix of tree structure.
\end{proof}

We  than have 
\vspace{0.3cm}


\begin{theorem}
 Let $A \in \bRn$ be non-singular, irreducible and acyclic. If $A$ is strictly row diagonally dominant then the sequence of the $|d_i|$ in 
 $(\ref{def-neu})$ is decreasing. If $A$ is strictly column diagonally dominant  then  the sequence of the $|f_i|$ in 
 $(\ref{def-neu})$ is decreasing.
 \end{theorem}
\begin{proof}
 The statements follow  from Theorem \ref{theo:det1}, since in (\ref{def-neu}) the entries of first row and column of $T$ are all the same. The same holds for  $U$, see Proposition   \ref{propo-one}.
\end{proof}

 The above Theorem  generalizes  Theorem 3.2  in \cite{Nab99b}. Here  the $f_i$ are  just  the $v_i$ and the $d_i$ are  just  the $y_i$.

\section{Examples} \label{sec:examples}

 The  next examples illustrate the decay. We  start with a tridiagonal matrix, a class of matrices for which the decay is well understood. But here  we consider  this matrix as a special  acyclic matrix. 
 
 \beqo
 A = \frac{1}{100}\left[ \begin{array}{cccccc}
 11 &    -1 &     0 &     0 &     0 &     0 \\
-1 &    12 &    -2 &     0 &     0 &     0 \\
     0 &    -2 &    13 &    -3 &     0 &     0  \\
     0  &    0 &    -3 &    14 &    -4  &    0 \\
     0  &    0 &     0 &    -4 &    15  &   -5  \\
     0 &     0 &     0  &    0 &    -5 &    16
\end{array} \right].
  \eeqo
 The graph of $A$  is just a line. 
 
  \begin{center}
 {\rm 
\noindent
\unitlength0.8cm
\begin{picture}(8,4)
 \put(2,2){\line(1,0){2.0}}
 \put(4,2){\circle*{0.2}}
 \put(2,2){\circle*{0.2}}
\put(0,2){\circle*{0.2}}
\put(10,2){\circle*{0.2}}
\put(6,2){\circle*{0.2}}
\put(8,2){\circle*{0.2}}
\put(0,2){\line(1,0){2.0}}
\put(4,2){\line(1,0){2.0}}
\put(6,2){\line(1,0){2.0}}
\put(8,2){\line(1,0){2.0}}
\put(-0.1,1.1){\makebox(0,0)[lb]{{\bf 1}}}
\put(1.9,1.1){\makebox(0,0)[lb]{{\bf 2}}}
\put(3.9,1.1){\makebox(0,0)[lb]{{\bf 3}}}
\put(5.9,1.1){\makebox(0,0)[lb]{{\bf 4}}}
\put(7.9,1.1){\makebox(0,0)[lb]{{\bf 5}}}
\put(9.9,1.1){\makebox(0,0)[lb]{{\bf 6}}}
 \end{picture}}
\end{center}
\vspace{0.5cm}
\begin{center}
{\sc Figure 3} 
\end{center}

 The inverse of $A$ is given by
 \beqo
 C = A^{-1} = \left[ \begin{array}{cccccc}
   9.1623 &    0.7848 &    0.1276 &    0.0299 &    0.0089     & 0.0028\\
    0.7848  &   8.6327 &    1.4040 &    0.3288 &    0.0979    & 0.0306  \\
    0.1276  &   1.4040 &    8.3602 &    1.9580 &    0.5828 &     0.1821 \\
    0.0299  &   0.3288 &    1.9580 &    8.2654 &    2.4604 &     0.7689  \\
    0.0089  &   0.0979 &    0.5828 &    2.4604 &    8.1743 &     2.5545  \\
    0.0028  &   0.0306 &    0.1821  &   0.7689 &    2.5545 &    7.0483
\end{array} \right]. 
  \eeqo

We can observe a clear decay along a row away  from the diagonal. But  this means for every vertex $i \in \{1, \ldots ,6\}$ we can  follow the  path to vertex $1$ or vertex $6$  and see the decay in the matrix along this path. Thus, for example 
consider vertex $4$.
 We have the path
 \beqo
 4 \rightarrow 3 \rightarrow 2 \rightarrow 1. 
 \eeqo
 Thus   we have  a decay of the entries
 \beqo
 C_{4,4}\geq C_{4,3} \geq C_{4,2} \geq C_{4,1},
 \eeqo
 namely
 \beqo
  8.2654 > 1.9580 > 0.3288 >  0.0299. 
 \eeqo
 But  we have also the path 
  \beqo
 4 \rightarrow 5 \rightarrow 6. 
 \eeqo
 Thus   we have  a decay of the entries
 \beqo
 C_{4,4} \geq C_{4,5} \geq C_{4,6},
 \eeqo
 namely
  \beqo
  8.2654 >   2.4604 > 0.7689. 
  \eeqo

 \vspace{0.5cm}
 
 Now, consider a matrix whose  graph  is the tree as in Figure 2.
 
 {\rm 
\noindent
\unitlength0.8cm
\begin{picture}(12,5)
\put(5.5,0){\circle*{0.2}}
\put(7.5,0){\circle*{0.2}}
\put(3.5,0){\circle*{0.2}}
\put(4.5,2){\circle*{0.2}}
\put(6.5,2){\circle*{0.2}}
\put(5.5,4){\circle*{0.2}}
\put(5.8,0){\makebox(0,0)[lb]{{\bf 5}}}
\put(3.8,0){\makebox(0,0)[lb]{{\bf 3}}}
\put(4.7,2){\makebox(0,0)[lb]{{\bf 2}}}
\put(7.8,0){\makebox(0,0)[lb]{{\bf 6}}}
\put(6.8,2){\makebox(0,0)[lb]{{\bf 4}}}
\put(5.8,4){\makebox(0,0)[lb]{{\bf 1}}}
\put(5.5,0){\line(1,2){1}}
\put(7.5,0){\line(-1,2){1}}
\put(3.5,0){\line(1,2){1}}
\put(4.5,2){\line(1,2){1}}
\put(6.5,2){\line(-1,2){1}}
\end{picture}
\vspace{.05in}

\begin{center}
{\sc Figure 2} 
\end{center}
}

We  choose
 \beqo
 A = \frac{1}{100} \left[ \begin{array}{cccccc}
 11 &    -5  &    0 &    -3 &     0 &     0  \\
    -1 &    12 &    -1 &     0 &     0 &     0  \\
     0 &    -2 &    13  &    0 &     0 &     0  \\
    -3 &     0  &    0 &   14 &    -2 &    -2  \\
     0 &    0  &    0 &    -4 &    15 &     0 \\
     0 &     0  &    0  &   -5 &     0 &    16
\end{array} \right].
  \eeqo
 
 The inverse is given by 
 
 \beqo
 C = A^{-1} = \left[ \begin{array}{cccccc}
  10.1245 &   4.2733 &   0.3287 &   2.3652 &   0.3154 &   0.2957 \\
    0.8547 &   8.8023 &   0.6771 &   0.1997 &   0.0266 &   0.0250 \\
    0.1315 &   1.3542 &   7.7965 &   0.0307 &   0.0041 &   0.0038 \\
    2.3652 &   0.9983 &   0.0768 &   8.3397 &   1.1120  &  1.0425 \\
    0.6307 &    0.2662 &    0.0205 &   2.2239 &   6.9632 &   0.2780 \\
    0.7391 &    0.3120 &    0.0240 &    2.6062 &    0.3475 &    6.5758 
 \end{array} \right].
  \eeqo

 Obviously, there is no decay away  from the diagonal entries. But Theorem \ref{theo-main6}   gives the hidden decay. 
 Now, have a look  at that  decay. Consider vertex $6$. 
 We have the path
 \beqo
 6 \rightarrow 4 \rightarrow 1 \rightarrow 2 \rightarrow 3. 
 \eeqo
 Thus,   we have  a decay of the entries
 \beqo
 C_{6,6} \geq C_{6,4} \geq C_{6,1}  \geq C_{6,2} \geq C_{6,3},
 \eeqo
 namely
 \beqo
  6.5758 > 2.6062 >  0.7391 > 0.3120 > 0.0240.
 \eeqo

Or  consider vertex $2$. 
 We have the path
 \beqo
 2 \rightarrow 1 \rightarrow 4 \rightarrow 5. 
 \eeqo
 Thus,   we have  a decay of the entries
 \beqo
  C_{2,2} \geq C_{2,1} \geq C_{2,4} \geq C_{2,5}
 \eeqo
 namely
 \beqo
  8.8023 > 0.8547 > 0.1997 > 0.0266.
 \eeqo

 Next, we consider a matrix that  is strictly diagonally dominant by rows but not by columns. The graph is the same  as above.

\beqo
 A = \frac{1}{100} \left[ \begin{array}{cccccc}
 11 &    -5  &    0 &    -3 &     0 &     0  \\
    -1 &   4 &    -2 &     0 &     0 &     0  \\
     0 &    -2 &    13  &    0 &     0 &     0  \\
    -3 &     0  &    0 &   14 &    -2 &    -2  \\
     0 &    0  &    0 &    -4 &    15 &     0 \\
     0 &     0  &    0  &   -5 &     0 &    16
\end{array} \right].
  \eeqo
The inverse is given by 
\beqo
 C = A^{-1} = \left[ \begin{array}{cccccc}
 11.2939 &   15.2938 &    2.3529 &    2.9797 &    0.7946  &   0.9312 \\
    3.0588 &   31.2254 &    4.8039 &    0.8070 &    0.2152 &    0.2522 \\
    0.4706 &   4.8039 &   8.4314 &   0.1242 &   0.0331 &   0.0388 \\
    2.9797 &    4.0350 &   0.6208 &   9.5806 &   2.5548 &   2.9939 \\
    0.7946 &    1.0760 &   0.1655 &   2.5548 &   7.3480 &   0.7984 \\
    0.9312 &   1.2609 &   0.1940 &   2.9939 &   0.7984 &   7.1856 
 \end{array} \right].
  \eeqo
Consider again the path 
 \beqo
 6 \rightarrow 4 \rightarrow 1 \rightarrow 2 \rightarrow 3. 
 \eeqo
 But we just have 
 \beqo
 C_{6,6} \geq C_{6,4} \geq C_{6,1}  \leq C_{6,2} \geq C_{6,3},
 \eeqo
 namely
 \beqo
  7.1856 > 2.9939 > 0.9312 < 1.2609 > 0.1940.
 \eeqo 
  This can happen, since we do not have row diagonal dominance in row 2. 
  Note that we have  a decay of the entries
 \beqo
  C_{2,2} \geq C_{2,1} \geq C_{2,4} \geq C_{2,5}.
 \eeqo
 Thus, it is the missing row diagonally dominance in row 2 and not the 'ups' and 'downs'  in the graph of $A$ which destroys the decay.

 \begin{center}
  {\bf Acknowledgment}
   \end{center}
  The author would like to thank the referees for reading the  paper with great care and for their suggestions which helped to improve the paper. 

\bibliographystyle{abbrv}
\bibliography{acyclic.bib}

\end{document}